\documentclass{article}
\usepackage{amsmath,amsthm,amssymb,latexsym}
\usepackage[active]{srcltx}

\title{{\bf The Grushin plane and quasiconformal Jacobians}}
\author{\Large {William Meyerson}}
\date{}
\newtheorem{teo}{Theorem}

\newtheorem{lemma}[teo]{Lemma}
\newtheorem{defi}[teo]{Definition}

\newtheorem{prop}[teo]{Proposition}

\newtheorem{que}[teo]{Question}

\begin{document}
\maketitle
\begin{abstract}
We construct a quasiconformal map from the Grushin plane to the Euclidean plane.  Then, we generalize the Grushin plane and explain how the Grushin plane and its generalizations can serve as intermediaries in dealing with quasiconformal maps on Euclidean spaces.  In particular, we construct a family of quasiconformal embeddings of the Euclidean plane into larger Euclidean spaces whose Jacobians fail to be locally integrable on a line.  
\end{abstract}

\begin{defi}

The \textbf{Grushin plane} $G$  is the metric completion of the Riemannian metric space
$\{(x, y) \in \textbf{R}^2: x \neq 0\}$ equipped with Riemannian metric $ds^2 = dx^2 + x^{-2} dy^2$.
\end{defi}

This Riemannian metric naturally gives $G$ the coordinate structure of $\textbf{R}^2$ along with a   geodesic structure.  This geodesic structure makes $G$ a metric space under the Carnot-Carath\'eodory distance.  We define a \textbf{horizontal} vector to be any vector in $\textbf{R}^2$ of the form $(t, 0)$ where $t$ is real, and we define the Carnot-Carath\'eodory distance as follows:  

\begin{defi}
Given two points $g, h \in G$, let $\Gamma_{g,h}$ be the set of all curves

$$\gamma:  [0, 1] \rightarrow G$$
with $\gamma(0) = g$, $\gamma(1) = h$, and $\dot{\gamma}(t)$ is horizontal
for each $t \in [0, 1]$ with $\gamma(t)$ on the $y$-axis.  Then the \textbf{Carnot-Carath\'eodory distance} from $g$ to $h$ is
$$d_{CC}(g, h) = \inf_{\gamma \in \Gamma_{g, h}} \int_0^1 
\sqrt{(\dot{x}(t))^2 + x(t)^{-2} (\dot{y}(t))^2} dt$$
(where $\gamma(t) = (x(t), y(t))$ for each $t$).
\end{defi}

\noindent
Note that a rectifiable curve must have horizontal tangent at each point where it 
crosses the $y$-axis.

\begin{defi}
For $z_1, z_2 \in G$, the \textbf{Grushin plane quasidistance} $d(z_1, z_2)$ is defined as follows:  
$$d(z_1, z_2) := \textrm{max}\{|x_1 - x_2|, \textrm{min}\{\sqrt{|y_1, - 
y_2|}, \frac{|y_1 - y_2|}{\max\{|x_1|, |x_2|\}}\}\}$$
where $z_1 = (x_1, y_1)$ and $z_2 = (x_2, y_2)$ with $x_i, y_i \in \textbf{R}$.
\end{defi}

\noindent
It is known (cf \cite{Be}) that there exists $C > 1$ such that for each $z_1, z_2 \in G$,

$$C^{-1} d_{CC}(z_1, z_2) \leq d(z_1, z_2) \leq C d_{CC}(z_1, z_2).$$
From this inequality, it is clear that except on the vertical axis, the Grushin plane is locally 
bi-Lipschitz to Euclidean space (but with a constant that blows up as we 
get closer to the axis).  However, the distance between two points on the 
vertical axis is proportional to the square root of their Euclidean distance.

In other words, the Grushin plane is a union of a 
(disconnected) Riemannian manifold and a line of Hausdorff dimension two, 
making it a sub-Riemannian manifold of both Euclidean and Hausdorff 
dimension two.

To motivate our results, a few more definitions are needed.

\begin{defi}
Let $(X, d)$ and $(Y, d')$ be metric spaces and $f: X \rightarrow Y$.    Also let $\eta: (0, \infty) \rightarrow (0, \infty)$ be an increasing homeomorphism.  If for all $z_1, z_2, z_3 \in X$,
$$\frac{d_Y(f(z_3), f(z_1))}{d_Y(f(z_2), f(z_1))} \leq 
\eta(\frac{d_X(z_3, z_1)}{d_X(z_2, z_1)})$$
then we say $f$ is \textbf{$\eta$-quasisymmetric}.
\end{defi}

\begin{defi}
Let $(X, d)$ and $(Y, d')$ be metric spaces and $f: X \rightarrow Y$.   If there exists $C > 0$ such that for all $z_1, z_2, z_3 \in X$ with $d_X(z_3, z_1) \leq d_X(z_2, z_1)$, $d_Y(z_3, z_1) \leq C d_Y(z_2, z_1)$ then we say $f$ is \textbf{$C$-weakly quasisymmetric}.
\end{defi}

\begin{defi}
 A map $F:  \textbf{R}^2 \rightarrow \textbf{R}^N$ is \textbf{absolutely continuous on lines} (ACL) if, for each closed rectangle $R$ in the domain of $F$ with sides parallel to the coordinate axes, $F|R$ is absolutely continuous on almost all line segments parallel to the sides of $R$.
\end{defi}

\begin{defi}
The \textbf{first Heisenberg group} $H_1$ is the set 

$$\{(z, t): z \in \textbf{C}, t \in \textbf{R}\}$$

\noindent equipped with the following group law:

$$(z, t)(w, s) = (z + w, t + s - \frac{1}{2} \Im(z\bar{w}))$$

\noindent where $\Im$ denotes imaginary part.
\end{defi}

Although by \cite{Ar} $G$ is a metric quotient of $H_1$, the two spaces $G$ and $H_1$ can have very different metric properties.  For example (cf \cite{M}), if $\phi$ is a Lipschitz map from any bounded subset of $H_1$ to any Euclidean space, the domain of $\phi$ can be partitioned into a union of sets $F_1, \dots, F_k, Z$ such that $\phi|F_k$ is bi-Lipschitz and $\phi(Z)$ has arbitrarily small Hausdorff content with respect to the Hausdorff dimension of $H_1$ (which is 4).  However, the same does not hold for the Grushin plane.

Also, $H_1$ does not embed quasisymmetrically into any Euclidean space (see, for example, \cite{P}).  By contrast, for $G$ we have the following result:

\begin{teo}
Let $F:  G \rightarrow \textbf{R}^2$ be defined as follows:   $F(x, y) = (x |x|, 
y)$.  Then $F$ is quasisymmetric.
\end{teo}

We begin by showing $F$ is weakly quasisymmetric.  This will be an immediate consequence of the following proposition:

\begin{prop}
There exists $C > 0$ such that for each $z \in G$, there exists an increasing function $f_z:  \textbf{R}_+ \rightarrow \textbf{R}_+$ such that whenever $z' \in G$, 
$$C^{-1} f_z(d(z, z')) \leq |F(z) - F(z')| \leq C f_z(d(z, z')).$$
\end{prop}

\begin{proof}
Because $d_{CC}$ and $d$ are invariant under both vertical translation and reflection in the $y$ axis, we may assume $z = 
(x, 0)$ with $x \geq 0$.

Also, we may assume that the absolute value signs denote distance with respect to the $L^\infty$ norm on $\textbf{R}^2$ (this does not increase any distance, nor does it decrease distances by more than a multiplicative factor of $\sqrt 2$).

We then complete the proof by considering several separate cases which encompass every resulting possibility.

\textbf{Case 1:  $x = 0$, i.e. $z$ is the origin.} In this case, we note that if $z' = (x', y')$ satisfies $\sqrt{|y'|} 
\geq \frac{|y'|}{|x'|}$, then $|x'| \geq \sqrt{|y'|}$ so that 
$$d(z, z') = \textrm{max}\{|x'|, \sqrt {|y'|}\}.$$
Since $F(z) = (0, 0)$ and $F(z') = (x' |x'|, y')$, we note that 
$$|F(z') - F(z)| = \textrm{max}\{|x'|^2, |y'|\} = d(z, z')^2$$ and we can take 
$f_z(r) = r^2$ for all such $z$.

\textbf{Case 2:  $x > 0$ and $d(z, z') \leq 3x$.} Consider $z' = (x', y')$ with $d(z, z') = r$.  Since $|x'| \leq 4 |x|$ and 
$\sqrt {|y'|} \leq 4x$, we have
$$\frac{|y'|}{\textrm{max}\{|x|, |x'|\}} \leq \frac{|y'|}{x} \leq 4 
\frac{|y'|}{\textrm{max}\{|x|, |x'|\}}$$
while $\frac{|y'|}{|x|} \leq 4 \sqrt {|y'|}$
so that up to a multiplicative factor of 4, we can take
$$d(z, z') = \textrm{max}\{|x - x'|, \frac{|y'|}{x}\}.$$
In this case, $$|F(z') - F(z)| = \textrm{max}\{x |x| - x' |x'|, |y'|\}.$$
Note that for $x' \geq 0$, $x |x| - x' |x'| = x^2 - x'^2 = (x - x')(x + 
x')$ while for $x' < 0$, $x |x| - x' |x'| \in (x^2, 5x^2]$ as $x' \geq 
-2x$.  In any event, $\frac{|x |x| - x' |x'| |}{|x| |x - x'|} \in [\frac{1}{3}, 5]$ whenever $x \neq x'$.  Similarly, $|y| = \frac{|y|}{x} * |x|$ so here we take $f_z(r) = |x|r$ and 
note that $|F(z') - F(z)| \in [\frac{1}{3} f_z(r), 5 f_z(r)]$ (and therefore, if 
we were to use our original quasidistance instead, 
$|F(z') - F(z)| \in [\frac{1}{12} f_z(r), 20 f_z(r)]$).

\textbf{Case 3: $x > 0$ and $d(z, z') > 3x$.}  Writing $z' = (x', y')$ with $d(z, z') = r$ we now consider three subcases.

\textbf{Case 3.1:  $|x'| \leq x$.}  Here, we note that quasidistance cannot be reached via $|x - x'|$ (as $|x - x'| \leq 2|x|$), so that $$d(z, z') = \textrm{min}\{\sqrt{|y'|}, \frac{|y'|}{x}\}.$$  Further, if 
$\sqrt {|y'|} \geq \frac{|y'|}{x}$, then $x \geq \sqrt {|y'|}$ so $d(z, 
z') \leq x$ producing a contradiction.  Therefore, we conclude $$d(z, z') = 
\sqrt {|y'|}.$$
This means that $\sqrt {|y'|} > 3x$; writing $F(z) = (w_1, w_2)$ and $F(z') = (w'_1, w'_2)$, we note
$|w_1 - w'_1| \leq 2x^2$ while $|w_2 - w'_2| = |y'| > 9x^2$.  Therefore, $|F(z) - F(z')| = (y')^2$.  
As $(y')^2 = r^2$ in this case, our candidate for $f_z(r)$ for this subcase is $r^2$.

\textbf{Case 3.2:  $\frac{r}{3} > |x'| > |x|$.}  By switching the role of $z$ and $z'$ (and translating by $(0, -y')$), we reduce to Case 3.1 with the same value of $r$ so we take $f_z(r) = r^2$.

\textbf{Case 3.3:  $|x'| \geq \frac{r}{3}$.} By switching the role of $z$ and $z'$ (and translating by $(0, -y')$), we 
reduce to Case 2 with the same value of $r$.  Because $$|x'| \leq |x| + r < 2r,$$ 
the value of $f_{z'}(r)$ created in Case 2, $|x'|r$, is in the range $[\frac{r^2}{3}, 2r^2]$.
Therefore, we can use $f_z(r) = r^2$ here too.

As every possibility is accounted for, the proposition is shown.
\end{proof}

\begin{proof}[Proof of theorem]
This follows by noting that with the $f_z$ functions constructed above, $f_z(\epsilon r) \leq \epsilon f_z(r)$ for all $\epsilon \in (0, 1)$.
\end{proof}

We can generalize the above result by slightly changing the metric space on the Grushin plane.

\begin{defi}

Let $\alpha > 0$.  The \textbf{generalized Grushin plane} $G_\alpha$  is defined as the metric completion of the Riemannian metric space $\{(x, y) \in \textbf{R}^2: x \neq 0\}$ equipped with Riemannian metric $ds^2 = dx^2 + |x|^{-\alpha} dy^2$.
\end{defi}

Once again, this Riemannian metric naturally gives $G_\alpha$ the coordinate structure of $\textbf{R}^2$ along with a geodesic structure.  Note that $G_2$ is what was originally called $G$.  This geodesic structure makes $G_\alpha$ a metric space under the Carnot-Carath\'eodory distance, defined below: 

\begin{defi}
Given two points $g, h \in G_\alpha$, let $\Gamma_{g,h}$ be the set of all curves

$$\gamma:  [0, 1] \rightarrow G_\alpha$$
with $\gamma(0) = g$, $\gamma(1) = h$, and $\dot{\gamma}(t)$ is horizontal
for each $t \in [0, 1]$ with $\gamma(t)$ on the $y$-axis.  Then the \textbf{Carnot-Carath\'eodory distance} from $g$ to $h$ is
$$d_{CC}(g, h) = \inf_{\gamma \in \Gamma_{g, h}} \int_0^1 
\sqrt{(\dot{x}(t))^2 + |x(t)|^{-\alpha} (\dot{y}(t))^2} dt$$
(where $\gamma(t) = (x(t), y(t))$ for each $t$).
\end{defi}

\begin{defi}
Given $z_1, z_2 \in G_\alpha$, the \textbf{Grushin plane quasidistance} $d(z_1, z_2)$ is defined as follows:  

$$d(z_1, z_2) := \textrm{max}\{|x_1 - x_2|, \textrm{min}\{|y_1 - 
y_2|^{\frac{2}{2 + \alpha}}, \frac{|y_1 - y_2|}{\max\{|x_1|^{.5 \alpha}, 
|x_2|^{.5 \alpha}\}}\}\}$$ where $z_1 = (x_1, y_1)$ and $z_2 = (x_2, y_2)$ with $x_i, y_i \in \textbf{R}$.
\end{defi}

\begin{lemma}
There exists $C > 1$ (dependent on $\alpha$) such that for each $z_1, z_2 \in G$,
$$C^{-1} d_{CC}(z_1, z_2) \leq d(z_1, z_2) \leq C d_{CC}(z_1, z_2).$$
\end{lemma}

\begin{proof}
It suffices to consider the case where $z_1 = (x, 0)$ and $z_2 = (x, \epsilon)$ for $x, \epsilon > 0$.  In this case, basic geometric reasoning shows that the optimal path joining $z_1$ and $z_2$ can be approximated  (multiplying path length by at most a constant dependent on $\alpha$) by a horizontal line following $(x, 0)$ to $(x + t, 0)$
followed by a vertical line up to $(x + t, \epsilon)$ and another horizontal line back to $(x, \epsilon)$ for some $t > 0$.  Each horizontal line has length $t$ and the vertical line has length $\frac{\epsilon}{(x + t)^{\frac{\alpha}{2}}}$ making the total length
$$2t + \frac{\epsilon}{(x + t)^{\frac{\alpha}{2}}}.$$
Differentiating gives 
$$2 - .5 \alpha \epsilon (x + t)^{-.5 \alpha - 1} = 0$$ so that 
$$x + t = K \epsilon^{\frac{2}{2 + \alpha}}$$ where 
$$K = (\frac{\alpha}{4})^{\frac{2}{2 + \alpha}}$$ is independent of $\epsilon$.

If $t > 0$ then the approximation to the optimal geodesic is indeed found via the method described:  as 
$$2t \leq 2K \epsilon^{\frac{2}{2 + \alpha}}$$ and 
$$\frac{\epsilon}{(x + t)^{\frac{\alpha}{2}}} = K^{\frac{-\alpha}{2}} \epsilon^\frac{2}{2 + \alpha}$$
the path length in question is indeed comparable to $$\epsilon^\frac{2}{2 + \alpha};$$
this is equal to $d((x, 0), (x, \epsilon))$ up to a multiplicative constant because 
$$\frac{\epsilon}{x^{.5 \alpha}} \geq K^{-.5 \alpha} \epsilon^\frac{2}{2 + \alpha}$$
since $$x \leq K \epsilon^{\frac{2}{2 + \alpha}}.$$

If $t \leq 0$ then the method described would involve moving closer to the vertical axis 
(which is inefficient as this increases vertical length) so the approximation to the optimal geodesic is simply a vertical line joining
$(x, 0)$ to $(x, \epsilon)$ which is clearly of length $\frac{\epsilon}{x^{.5 \alpha}}.$  This is equal to $d((x, 0), (x, \epsilon))$ up to a multiplicative constant because 
$$\frac{\epsilon}{x^{.5 \alpha}} \leq K^{-.5\alpha} \epsilon^\frac{2}{2 + \alpha}$$
$$\textrm{since } x \geq K \epsilon^{\frac{2}{2 + \alpha}}.$$

As the multiplicative constants involved only depend on $\alpha$, the lemma is shown.
\end{proof}

\begin{teo}
Let $F:  G_\alpha \rightarrow \textbf{R}^2$ be defined as follows:   $$F(x, y) = (x |x|^{.5 \alpha}, 
y).$$  Then $F$ is quasisymmetric.
\end{teo}

\begin{proof}
The proof is the same as the proof of Theorem 8, although the multiplicative constants depend on $\alpha$.
\end{proof}

In the final part of this paper, we use Theorem 14 as a base map to investigate what is commonly known as the \emph{quasiconformal Jacobian problem}.  One version of the quasiconformal Jacobian problem, as stated in \cite{Bo}, is the following:

\begin{que}
Let $n \geq 2$.  For which locally integrable functions $w \geq 0$ on $\textbf{R}^n$ does there exist a quasiconformal mapping $f:  \textbf{R}^n \rightarrow \textbf{R}^n$ and a constant $C \geq 1$ such that $$\frac{1}{C} w(z) \leq J_f(z) \leq Cw(z)$$
for almost all $z \in \textbf{R}^n$, where $J_f(z) = \textrm{det}(Df(z))$?
\end{que}
\noindent
Another version, as stated in \cite{Sem}, reads as follows:

\begin{que}
Let $n \geq 2$ and $\mu$ be a positive measure on $\textbf{R}^n$.  Define a quasidistance $\delta$ on $\textbf{R}^n \times \textbf{R}^n$ by

$$\delta(z, z') = (\int_{B(z, |z - z'|)} d\mu)^{\frac{1}{n}}.$$
Under which conditions on $\mu$ is $\mu$ comparable in size to the Jacobian of a quasiconformal map from $\textbf{R}^n$ to a larger Euclidean space?  Equivalently, when is $\delta$ bi-Lipschitz equivalent to the standard Euclidean metric on some larger Euclidean space?

\end{que}

Note that the above questions, as stated, only consider locally integrable weights.  For example, in the planar case $n = 2$, the function $f:  \textbf{R}^2 \rightarrow \textbf{R}$ with $f(x, y) = |x|^{\beta}$ can only give a locally integrable measure of the form $\mu = f d\lambda$ (where $\lambda$ is Lebesgue measure) for $\beta \geq -1$.  This motivates the following question:  

\begin{que}
Suppose $\beta \in \textbf{R}$.  Does there exist a quasisymmetric map $H$ from $\textbf{R}^2$ to some Euclidean space whose Jacobian is comparable to $|x|^\beta$ at almost every point $(x, y) \in \textbf{R}^2$?
\end{que}

For the case $\beta > -1$, when $|x|^\beta$ is locally integrable, the result is due to Semmes (\cite{Sem}):  the answer is yes for $-1 < \beta \leq 0$ and no for $\beta > 0$.

Further, we shall show that the answer is yes for the case where $-2 < \beta \leq -1$.  To find this, one additional map from the recent literature is needed as well. 

\begin{teo}
For $\alpha > 0$ there exists a biLipschitz $\phi_\alpha:  G_\alpha \rightarrow \textbf{R}^{n_\alpha}$ if $n_\alpha$ is sufficiently large.
\end{teo}
\noindent
This is a direct consequence of Theorem 1.1 in \cite{Seo}.  Armed with these $\phi_\alpha$, we can now state and prove another theorem.

\begin{teo} 
Suppose $-2 < \beta \leq 0$.  Then there exists a quasisymmetric map $\Phi_\beta$ from $\textbf{R}^2$ to some Euclidean space whose Jacobian is comparable to $|x|^\beta$ almost everywhere.
\end{teo}

\begin{proof}
Because the $\beta = 0$ case is trivial (take $\Phi_0$ to be the identity), we assume $\beta < 0$ for our construction.  Setting $H_\alpha: \textbf{R}^2 \rightarrow \textbf{R}^{n_\alpha}$ as $H_\alpha = \phi_\alpha \circ F_\alpha^{-1}$, we note that $H_\alpha$ is 
clearly quasisymmetric (as a composition of quasisymmetries).  

To analyze $H_\alpha$, we note that locally $F_\alpha$ multiplies volume by $x^{.5 \alpha}$ up to a 
multiplicative constant, so $F_\alpha^{-1}$ (sending the $x$ coordinate 
to $|x|^{\frac{2}{2 + \alpha}}$) locally multiplies volume by 
$|x|^{-\frac {2 \alpha}{2 + \alpha}}$ and therefore $H_\alpha$ has 
Jacobian comparable to $|x|^{-\frac{2 \alpha}{2 + \alpha}}$.  Setting $\Phi_\beta = H_{\frac{-2\beta}{2 + \beta}}$ completes the proof.
\end{proof}

Only the case $\beta \leq -2$ remains open:  in other words, what remains unsolved is

\begin{que}
Let $\beta \leq -2$.  Does there exist a quasisymmetric map $H$ from $\textbf{R}^2$ to some Euclidean space whose Jacobian is comparable to $|x|^{-\beta}$ everywhere?
\end{que}

For one potential idea on how to analyze this, we note that for $-2 < \beta < 0$, the $\Phi_\beta$ constructed are absolutely continuous on lines by construction.  Further, if $F:  \textbf{R}^2 \rightarrow \textbf{R}^N$ had Jacobian comparable to $|x|^{-t}$ everywhere for some $t \geq 2$, then $F$ would not be absolutely continuous on lines because its directional derivative in the $x$ direction, on the order of $|x|^{-\frac{t}{2}}$ almost everywhere, could not be locally integrable on almost all horizontal lines.

Therefore, if $H$ in the above question were required to be absolutely continuous on lines, the answer would be no and our characterization would be complete.  However, by \cite{Bi} there exists a quasiconformal map from $\textbf{R}^2$ to $\textbf{R}^3$ which is \textbf{not} absolutely continous on lines, so we cannot assume \emph{a priori} that $H$ satisfies the ACL condition.

\end{document}